\documentclass[10pt,reqno]{amsart}
\usepackage[a4paper,top=3 cm,bottom=3 cm,left=2.5 cm,right=2.5 cm]{geometry}

\usepackage[utf8]{inputenc}  
\usepackage[T1]{fontenc}
\usepackage{pgfplots}
\usepackage{helvet}
\usepackage{graphicx}
\usepackage{xcolor}
\usepackage{amsmath,amssymb, mathtools}
\usepackage{hyperref}
\usepackage{cleveref}
\usepackage{dsfont}
\usepackage{todonotes}

\usepackage{subfig}

\DeclareMathOperator{\esssup}{ess\,sup}

\def\R{\mathbb R}
\def\N{\mathbb N}

\def\E{\mathbb E} 

\def\to{\rightarrow}
\def\eps{\varepsilon}

\newtheorem{theorem}{Theorem}[section]

\newtheorem{corollary}{Corollary}[section]
\newtheorem{lemma}{Lemma}[section]
\newtheorem{definition}{Definition}[section]

\newtheorem{remark}{Remark}[section]

\newtheorem{assumption}{Assumption}[section]

\usepackage{titletoc}

\makeatletter
\newcommand\thankssymb[1]{\textsuperscript{\@fnsymbol{#1}}}
\makeatother

\pgfplotsset{compat=1.18} 
\begin{document}
                
\title{Averaged observations and turnpike phenomenon for parameter-dependent systems}

\author{Mart\'in Hern\'andez}
\address{M. Hern\'andez, Chair for Dynamics, Control, Machine Learning, and Numerics, Alexander von Humboldt-Professorship, Department of Mathematics,  Friedrich-Alexander-Universit\"at Erlangen-N\"urnberg, 91058 Erlangen, Germany.}
\email{martin.hernandez@fau.de}

\author{Martin Lazar}
\address{M. Lazar, Department of Electrical Engineering and Computing, University of Dubrovnik, \'Cira Cari\'ca 4, 20000 Dubrovnik, Croatia.}
\email{mlazar@unidu.hr}

\author{Sebasti\'an Zamorano}
\address{S. Zamorano, Universidad de Santiago de Chile, Departamento de
Matem\'atica y Ciencia de la Computaci\'on, Facultad de Ciencia, Casilla 307-Correo 2,
Santiago, Chile.}
 \email{sebastian.zamorano@usach.cl}

\thanks{
}

% \thanks{\thankssymb{1} 
% add institution
% }


\begin{abstract} 
Our main contribution in this article is the achievement of the turnpike property in its integral and exponential forms for parameter-dependent systems with averaged observations in the cost functional. Namely, under suitable assumptions with respect to the matrices that defined the dynamics and the cost functional, we prove that the optimal control and state for the evolutionary problem converge in average to the optimal pair of an associated stationary problem. Moreover, we characterize the closeness between these two optimal solutions, proving that over a large time interval, they are exponentially close.
\end{abstract}

\maketitle

\tableofcontents

\section{Introduction}
\subsection{Main results}
Let us begin this article by presenting our main contributions. For it, let $T>0$, $n,\,m\in\N$, and let $(\Omega, \mathcal{F}, \mu)$ be a probability space. We consider the following optimal tracking control problem with averaged observations
\begin{align}\label{eq:evol_funtional}
    \min_{u\in L^2(0,T;\R^m)}\left\{J^T(u)=\frac{1}{2}\int_{0}^T\Big(\|u(t)\|_{\R^m}^2+\|\E[C(\cdot)x(t,\cdot)]-z\|^2_{\R^n}\Big) dt + \langle x(T,\cdot),\varphi_T(\cdot)\rangle_{L^2(\Omega;\R^n)}\right\},
\end{align}
subject to $x=x(t,\omega)\in\R^n$ solving the parameter-dependent evolutionary equation
\begin{align}\label{eq:evol_equation}
\begin{cases}    x_t(t,\omega)+A(\omega)x(t,\omega)=B(\omega)u(t),\quad t>0,\\
x(0,\omega)=x_0(\omega).
\end{cases}
\end{align}
Here $\omega\in\Omega$ corresponds to the random parameter, $z\in \R^n$ is a fixed target, $x_0\in L^\infty(\Omega;\R^n)$ denotes the initial condition, and $\varphi_T\in L^2(\Omega;\R^n)$ determines a linear regularization of the average of the final state. In addition, we consider the three random matrices $A,\,C\in C^0(\Omega,\mathcal{L}(\R^n))$ and $B\in C^0(\Omega,\mathcal{L}(\R^m;\R^n))$, constant in time and uniformly bounded with respect to $\omega$, in order to ensure the integrability of solutions with respect to the parameter. By $x=x(t,\omega)\in \R^n$, we denote the state of the equation, and $u=u(t)$ represents the control function, which is parameter-independent.
 
The space $L^2(\Omega;\R^n)$ is the set formed by the square-integrable functions with respect to the measure $\mu$, that is
\begin{align*}
    L^2(\Omega;\R^n):=\left\{x:\Omega\to\R^n\,\, \int_{\Omega}\|x(\omega)\|^2_{\R^n}d\mu(\omega)<\infty\right\},
\end{align*}
which is a Hilbert space endowed with inner product 
\begin{align*}
( x,y)_{L^2(\Omega; \R^n)}=\int_{\Omega}( x(\omega),y(\omega))_{\R^n}d\mu(\omega), \quad\forall x,y\in L^2(\Omega; \R^n).
\end{align*}
In a same manner, we can define the space $L^\infty(\Omega;\R^n)$ as
\begin{align*}
    L^\infty(\Omega;\R^n):=\left\{ x:\Omega\to \R^n\,:\, \esssup_{\omega\in \Omega}\|x(\omega)\|_{\R^n}<\infty\right\}.
\end{align*}
We also denote by $\E:L^2(\Omega;\R^n)\to\R^n$ the expectation operator i.e. 
\begin{align*}
    \E[x]:=\int_\Omega x(\omega) d\mu(\omega).
\end{align*}

It is immediate that $J^T$ given by \eqref{eq:evol_funtional} is a strictly convex, continuous and coercive functional. Therefore, existence and uniqueness of the optimal control for $J^T$, denoted by $u^T$, is a consequence of the direct method of calculus of variations. Let us denote by $x^T$ the optimal state, solution of \eqref{eq:evol_equation}, associated to the optimal control $u^T$. 

In this article, we are interested in the relationship between the optimal pair $(u^T,x^T)$ and the corresponding stationary optimal control. Namely, we want to analyze how $(u^T,x^T)$ is related, in an appropriate sense, to the optimal solution of the reference stationary optimal control problem
\begin{align}\label{eq:statio_functional}
    \min_{u\in \R^m}\left\{J^s(u,x)=\frac{1}{2}\left(\|u\|_{\R^m}^2+\|\E[C(\cdot)x(\cdot)]-z\|^2_{\R^n}\right)\right\},
\end{align}
subject to $x=x(\omega)\in \R^n$ be the solution of the parameter-dependent stationary equation 
\begin{align}\label{eq:statio_equation}
A(\omega)x(\omega)=B(\omega)u.
\end{align}
To establish how the two optimal solutions are related, we need to assume the following hypotheses related to the dynamics of \eqref{eq:evol_equation} and the cost functional, which are motivated by the notions of exponential stability and exponential detectability, classic notions in control and optimization theory (see, for instance, \cite{MR2273323}). 
\begin{assumption}\label{A1}
    There exists a feedback operator $K_C\in C^0(\Omega,\mathcal{L}(\R^n))$ uniformly bounded with respect to $\omega\in \Omega$ and a positive constant $\alpha_C>0$ such that
\begin{align*}
    (Av+ K_C\E[ K_C Cv],v)_{L^2(\Omega;\R^n)}\geq \alpha_C\|v\|_{L^2(\Omega;\R^n)}^2,\quad \text{for every }v\in L^2(\Omega;\R^n).
\end{align*}
\end{assumption}
\begin{assumption}\label{A2}
    There exists a feedback operator $K_B\in C^0(\Omega,\mathcal{L}(\R^m;\R^n))$ uniformly bounded with respect to $\omega\in \Omega$ and a positive constant $\alpha_B>0$ such that
\begin{align*}
    (A^*v+ K_B\E[K_B B^*v],v)_{L^2(\Omega;\R^n)}\geq \alpha_B\|v\|_{L^2(\Omega;\R^n)}^2,\quad \text{for every }v\in L^2(\Omega;\R^n).
\end{align*}
\end{assumption}
It is important to mention that, under Assumption \eqref{A1}, the existence and uniqueness of optimal solution for the minimization problem \eqref{eq:statio_functional} can be proved, cf. Theorem \ref{Th:uniq_statio_control}. Thus, we denote by $(u^s,x^s)$ this optimal solution.

Therefore, with all of these ingredients, we can state our first main result, which tells us that both the state and the evolutionary optimal control converge on average to their corresponding steady state and control. We called this result a \emph{integral turnpike property}.
\begin{theorem}\label{Th:integral_turnpike}
   Let us suppose that Assumptions \ref{A1} and \ref{A2} hold and $T>1$. Let $(u^T,x^T)$ and $(u^s,x^s)$ be the optimal pair of minimization problems \eqref{eq:evol_funtional} and \eqref{eq:statio_functional}, respetively. Then, 
   \begin{align*}
       \left\|\frac{1}{T}\int_0^T x^T(t,\cdot)dt - x^s(\cdot)\right\|_{L^2(\Omega;\R^n)}\to 0,\quad \text{ and }\quad
       \left\|\frac{1}{T}\int_0^T u^T(t)dt - u^s\right\|_{\R^m}\to 0,
   \end{align*}
as $T\to \infty$.
\end{theorem}

The proof of this integral turnpike property is based on energy estimates for the optimal state and the adjoint variable given by the Euler-Lagrange equation.

Under that the optimal solutions converge in average, we analyze the possibility to better describe this convergence in the following sense: Does there exist a time interval where the optimal solutions are exponentially close?  This property is known as the exponential turnpike phenomenon, which describes how the evolutionary optimal pair is made of three arcs. The first and the last are transient short-time arcs, and the middle piece is a long-time arc that stays exponentially close to the optimal steady-state of the associated static optimal control problem. In the next theorem, which is our second main result, we state this property in our setting.
\begin{theorem}\label{Th:exponential_turnpike}
Let the assumption of Theorem \ref{Th:integral_turnpike} be in force. Then, there exist two constants $K,\delta>0$ independent of $T$ such that 
\begin{align}
    \|x^T(t,\cdot)-x^s(\cdot)\|_{L^2(\Omega;\R^n)}+\|\varphi^T(t,\cdot)-\varphi^s(\cdot)\|_{L^2(\Omega;\R^n)}+\|u^T(t)-u^s\|_{\R^m}\leq K(e^{-t\delta}+e^{-(T-t)\delta}),
\end{align}
for every $t\in [0,T]$, where $\varphi^T=\varphi^T(t,\omega)\in \R^n$ and $\varphi^s=\varphi^s(\omega)\in\R^n$ denote the adjoint variables given by the first-order optimality condition using the Lagrangian approach for \eqref{eq:evol_funtional} and \eqref{eq:statio_functional}, respectively.
\end{theorem}

In this case, the proof will be based on energy estimates and ensuring that the operator corresponding to the Euler-Lagrange equations is bounded by a constant independent on the horizon time $T>1$.

The most important novelties and difficulties of this article can be described as follows:
\begin{itemize}
\item To the best of our knowledge, exponential turnpike property for cost functional as in this paper, i.e., with averaged observations, it is the first time that has been studied. Previous works consider more restrictive penalization, as in \cite{HLZ23}, where the authors studied the exponential turnpike property without the expectation and terminal constraint on the cost functional. Then, our Theorem \ref{Th:exponential_turnpike} can be seen as the first exponential turnpike result that extends the averaged controllability idea (see, for instance, \cite{MR4599945}) to optimal control.

\item In \cite{MR4533776}, an integral turnpike result was derived without imposing any assumption on the dynamics and in the cost functional. They studied the turnpike in terms of the controllability and observability of the Gramian operator. However, as we also derive an exponential turnpike, a different proof for the integral turnpike is needed in order to get that kind of result.

\item From the minimization problem for the stationary parameter-dependent problem \eqref{eq:statio_functional}, we can observe that the existence and uniqueness of optimal solution is not a direct consequence of the calculus of variations methodology. In order to establish the existence of the optimal control, from our perspective, Assumption \ref{A1} is the most adequate hypothesis.

\item Assumption \ref{A1} can be viewed as the natural generalization of the classical detectability condition (see, for instance, \cite{MR2273323}) for parameter-dependent equations and cost functional with averaged observations. From our perspective, this condition is much more natural than the one previously considered in \cite{HLZ23}. A discussion of this will be given at the end of the paper.

\item Thus, in this work we will call Assumption \ref{A2} as $(A,B)$ stabilizability on average and Assumption \ref{A1} as $(A,C)$ detectability on average. As far as we know, this is the first time when these names have appeared in the literature.

\item Without loss of generality, in this work, we have considered target function $z\in\R^n$ independent of the parameter $\omega\in\Omega$. As we will see in the proof of our main results, this assumption does not play any role in that.   
\end{itemize}

\subsection{Existing literature}
In this article, we focus our attention on the so-called integral and exponential turnpike property for parameter-dependent systems with averaged observations on the cost functional, cf. Theorem \ref{Th:integral_turnpike} and Theorem \ref{Th:exponential_turnpike}, respectively. There is a huge body of literature on the turnpike property, both the integral and the exponential. Let us mention some of the most relevant works from our perspective. The article \cite{porretta2013long} studied the phenomenon of turnpike for linear equations, both finite and infinite dimensions. The authors proved similar results as in this work, but assuming different kinds of hypotheses. For instance, in the finite-dimensional setting, they assumed control and observation conditions for the dynamical system and for the cost functional, respectively. Under the previous hypotheses, using the Riccati theory, they established the exponential turnpike property thanks to the exponential convergence of the matrix-solution of a Riccati differential equation to the corresponding algebraic one. For the integral turnpike result, the authors used in a strong manner the controllability assumption. We can mention that the proof done in \cite{porretta2013long} was the inspiration for our demonstration.

With stabilizability and detectability conditions, in \cite{grune2019sensitivity}, the authors have established a sensitivity and turnpike properties for infinite-dimensional equations. Contrary to the paper \cite{porretta2013long}, here, the exponential turnpike property was proved without the Riccati theory. Our approach follows some ideas from this article in order to obtain the exponential turnpike phenomenon. 

For parameter-dependent systems we found the work \cite{HLZ23}, where the authors showed the exponential and integral turnpike for the system as we, but with different cost functional and slightly different assumptions. Another work related to parameter-dependent systems is the article \cite{MR4533776}, in which an integral turnpike result for optimal averaged control problems was obtained without assumptions on stabilizability and detectability. 

Finally, we recommend the following survey article \cite{MR4436586} and the books \cite{zaslavski2006turnpike, zaslavski2019turnpike} and the references therein for a complete treatment of the turnpike topic.
%a non exhaustive list of articles related to the turnpike phenomenon can be summarized in \cite{breiten2020onthe, faulwasser2017onturnpike, grune2020exponential, MR4646523, MR4533776, hernandez2019greedy, porretta2016remarks, MR4201973, zamorano2018turnpike}.

\subsection{Organization of the paper}
The rest of the paper is organized as follows. In Section \ref{sec:2}, we state and prove some preliminary estimates for the solution of equation \eqref{eq:evol_equation} and the adjoint problem. In addition, we establish the existence and uniqueness of problem \eqref{eq:statio_functional}. Section \ref{sec:3} is devoted to the proof of Theorem \ref{Th:integral_turnpike} and Section \ref{sec:4} to the proof of Theorem \ref{Th:exponential_turnpike}. In Section \ref{sec:5}, we present a numerical example in order to validate numerically our main result concerning the exponential turnpike property. Finally, in Section \ref{sec:6}, we give a discussion about our main assumptions and present some open problems related to our work.

%%%%%%%%%%%%%%%%%%%%%%%%%

\section{Preliminary results}\label{sec:2}
The proof of our main results, namely Theorems \ref{Th:integral_turnpike} and \ref{Th:exponential_turnpike}, are based on a careful analysis and suitable estimates for the solution of the first-order optimality system associated to both problems \eqref{eq:evol_funtional}-\eqref{eq:evol_equation} and \eqref{eq:statio_functional}-\eqref{eq:statio_equation}.

\subsection{Time-dependent optimal control}
We start our analysis with the derivation of the optimality system for the minimization problem \eqref{eq:evol_funtional}. It is immediate that $J^T$ given by \eqref{eq:evol_funtional} is a strictly convex, continuous, and coercive functional. Therefore, existence and uniqueness of the optimal control for $J^T$, denoted by $u^T$, is a direct consequence of the direct method of calculus of variations. Let us denote by $x^T$ the optimal state associated to the optimal control $u^T$. 

Then, using the Euler-Lagrange approach, the optimal control $u^T$ can be characterized as
\begin{align}\label{eq:chara_evol_oc}
    u^T(t)=-\E[B^* \varphi^T(t,\cdot)], \quad t\in(0,T),
\end{align}
where $\varphi^T=\varphi^T(\omega,t)\in\R^n$ is the solution of the adjoint equation
\begin{align}\label{eq:evol_adjoint_equation}
    \begin{cases}
   -\varphi_t(t,\omega)+A^*(\omega)\varphi(t,\omega)= C^*(\omega)\left(\E[C(\cdot)x^T(t,\cdot)]-z\right),\quad t>0,\\
   \varphi(T,\omega)=\varphi_T(\omega),
    \end{cases}
\end{align}
with $x^T$ being the optimal solution of \eqref{eq:evol_equation} with control given by \eqref{eq:chara_evol_oc}.

Next, under the Assumption \ref{A1}, we can prove the following estimate for the solution of the equation \eqref{eq:evol_equation}.

\begin{lemma}\label{lem:inequality_state}
Let us assume that Assumption \ref{A1} is fulfilled. Then, there exists a constant $K_1>0$ independent of $t$ such that, for every $t\geq0$ and solution $x\in L^2(\Omega;\R^n)$ to \eqref{eq:evol_equation} we have
    \begin{align*}
        \|x(t,\cdot)\|_{L^2(\Omega;\R^n)}^2\leq K_1 \int_0^t\Big(\|u(s)\|^2_{\R^m}+\|\E[C(\cdot)x(s,\cdot)]\|_{\R^n}^2 \Big)ds+\|x_0(\cdot)\|^2_{L^2(\Omega;\R^n)}.
    \end{align*}
\end{lemma}

\begin{proof}
    Let us multiply equation \eqref{eq:evol_equation} by its solution $x$, integrate on $\R^n$ and over $\Omega$, and adding in both sides $(K_C\E[K_{C}Cx],x)_{L^2(\Omega;\R^n)}$, we have
    \begin{multline*}
        \frac{1}{2}\frac{d}{dt}\|x(t,\cdot)\|^{2}_{L^2(\Omega;\R^n)} + (A+K_{C}\E[K_{C}Cx(t,\cdot)],x(t,\cdot))_{L^2(\Omega;\R^n)}=(B(\cdot)u(t),x(t,\cdot))_{L^2(\Omega;\R^n)}\\+(K_C(\cdot)\E[K_{C}(\cdot)C(\cdot)x(t,\cdot)],x(t,\cdot))_{L^2(\Omega;\R^n)}.
    \end{multline*}
    Applying Young's inequality with  $\varepsilon_1,\,\varepsilon_2>0$ and using Assumption \ref{A1} we have 
    \begin{align*}
        \frac{d}{dt}\|x(t,\cdot)\|^{2}_{L^2(\Omega;\R^n)} +\|x(t,\cdot)\|^2_{L^2(\Omega;\R^n)}(2\alpha_C-\varepsilon_1-\varepsilon_2)\leq\frac{1}{\eps_1}\|B\|^2\|u(t)\|_{\R^n}^2+\frac{1}{\eps_2}\|K_C\|^4\|\E[C(\cdot)x(t,\cdot)]\|^2_{\R^n}.
    \end{align*}
    Taking $\varepsilon_1,\varepsilon_2$ such that $2\alpha_C>\varepsilon_1+\varepsilon_2$ and integrating over $(0,t)$, there exists a constant $K_1>0$ such that
    \begin{align*}
     \|x(t,\cdot)\|^{2}_{L^2(\Omega;\R^n)} +\int_0^t\|x(t,\cdot)\|^2_{L^2(\Omega;\R^n)}dt\leq K_1\int_0^t\Big(\|u(s)\|_{\R^n}^2+\|\E[C(\cdot)x(s,\cdot)]\|^2_{\R^n}\Big)ds+\|x_0(\cdot)\|^{2}_{L^2(\Omega;\R^n)},
    \end{align*}
    for every $t\geq 0$. Then, using the fact that the integral on the left-hand side is positive, the proof is finished.
\end{proof}

In a similar manner, but now using Assumption \ref{A2}, the following estimate for the solution of the equation \eqref{eq:evol_adjoint_equation} can be obtained.

\begin{lemma}\label{lem:inequality_adjoint}
Suppose the Assumption \ref{A2} is satisfied. Then, there exists a constant $K_2>0$ independent of $t$ such that, for every $0\leq t\leq T$ and  $\varphi$ solution to \eqref{eq:evol_adjoint_equation}, the following holds
    \begin{align*}
\|\varphi(t,\cdot)\|_{L^2(\Omega;\R^n)}^2\leq K_2\int_t^T\Big(\|\E[B^*(\cdot)\varphi(s,\cdot)]\|^2_{\R^m}+\|\E[C(\cdot)x(s,\cdot)]-z\|_{\R^n}^2 \Big) ds+\|\varphi_T(\cdot)\|^2_{L^2(\Omega; \R^n)}.
    \end{align*}
\end{lemma}

\begin{proof}
     Let us multiply equation \eqref{eq:evol_adjoint_equation} by its solution $\varphi$, integrate on $\R^n$ and over $\Omega$, and adding in both sides $(K_B\E[K_{B}B^*\varphi],\varphi)_{L^2(\Omega;\R^n)}$ we have
    \begin{multline*}
        -\frac{1}{2}\frac{d}{dt}\|\varphi(t,\cdot)\|^{2}_{L^2(\Omega;\R^n)} + (A^*+K_{B}\E[K_{B}B^*\varphi(t,\cdot)],\varphi(t,\cdot))_{L^2(\Omega;\R^n)}=(C^*(\cdot)\E[C(\cdot)x(t,\cdot)],\varphi(t,\cdot))_{L^2(\Omega;\R^n)}\\+(K_B(\cdot)\E[K_{B}(\cdot)B^*(\cdot)\varphi(t,\cdot)],\varphi(t,\cdot))_{L^2(\Omega;\R^n)}.
    \end{multline*}
    Applying Young's for $\varepsilon_1,\,\varepsilon_2>0$ and using Assumption \ref{A2} we have 
    \begin{multline*}
        -\frac{d}{dt}\|\varphi(t,\cdot)\|^{2}_{L^2(\Omega;\R^n)} +\|\varphi(t,\cdot)\|^2_{L^2(\Omega;\R^n)}(2\alpha_C-\varepsilon_1-\varepsilon_2)
        \leq\frac{1}{\eps_1}\|C^*\|^2\|\E[C(\cdot)x(t,\cdot)\|_{\R^n}^2\\+\frac{1}{\eps_2}\|K_B\|^4\|\E[B^*(\cdot)\varphi(t,\cdot)]\|^2_{\R^n}.
    \end{multline*}
    Choosing $\eps_1, \eps_2$ such that $2\alpha_B>\varepsilon_1+\varepsilon_2$ and integrating over $(t,T)$ there exists a constant $K_2>0$ such that
    \begin{align*}
     \|\varphi(t,\cdot)\|^{2}_{L^2(\Omega;\R^n)} +\int_t^T\|\varphi(s,\cdot)\|^2_{\R^n}dt\leq K_2\int_t^T\Big(\|\E[B^*(\cdot)\varphi(s,\cdot)]\|_{\R^n}^2+\|\E[C(\cdot)x(s,\cdot)]\|^2_{\R^n}\Big)ds\\+\|\varphi_T(\cdot)\|^{2}_{L^2(\Omega;\R^n)},
    \end{align*}
    for every $t\leq T$. Since the second term on the left-hand side is positive, we can conclude the proof.
\end{proof}

\begin{remark} Note that as a consequence of Assumptions \ref{A1} and \ref{A2}, we can directly conclude that, using Cauchy-Schwarz inequality on the left-hand side on both assumptions, the following inequalities, respectively
\begin{align}\label{eq:ineq_statio_1}
\alpha_C\|v\|_{L^2(\Omega:\R^n)}\leq \| Av\|_{L^2(\Omega;\R^n)}+\|\E[Cv]\|_{\R^n},   
\end{align}
and
\begin{align}\label{eq:ineq_statio_2}
\alpha_B\|v\|_{L^2(\Omega:\R^n)}\leq \| A^*v\|_{L^2(\Omega;\R^n)}+\|\E[B^*v]\|_{\R^n},
\end{align}
for some constants $\alpha_C,\,\alpha_B>0$ depending only on $K_C$ and $K_B$, respectively. In particular, the first inequality will be of utmost importance for the proof of the integral turnpike.
\end{remark}

%%%%%%%%%%%%%%%%%%%%%%%%%
\subsection{Stationary optimal control}
The stationary control problem \eqref{eq:statio_functional} needs a slightly more delicate and detailed treatment to establish the existence and uniqueness of the optimal control.

As mentioned in the Introduction, Assumption \ref{A1} is fundamental to obtain the following result concerning the existence and uniqueness of the optimal control problem \eqref{eq:statio_functional}-\eqref{eq:statio_equation}.

\begin{theorem}\label{Th:uniq_statio_control}
Let us suppose that Assumption \ref{A1} holds. Then, there exists a unique optimal control $u^s$ and optimal state $x^s$ of \eqref{eq:statio_functional}. 
\end{theorem}
\begin{proof}
    Let us denote by 
    \begin{align*}
        \mathcal{S}:=\{(u,x)\in \R^m\times L^2(\Omega;\R^n) \,:\, A(\omega)x(\omega)=B(\omega)u\}.
    \end{align*}
    Observe that the set $\mathcal{S}$ is closed (since $A,B$ are continuous), convex, and nonempty, indeed $(0,0)\in \mathcal{S}$. Thus, $argmin(J^s)\neq \emptyset$ (see \cite[Theorem 2.19. ]{MR3310025}). Let us consider a minimization sequence $\{u_n\}_{n\geq1}$ such that 
    \begin{align}\label{eq:minimization_sec_def}
        \lim_{n\to\infty} J^s(u_n)=\min_{u\in \R^m}J^s(u).
    \end{align}
Let us denote by $\{x_n\}_{n\geq 1}$ the associated state of each $u_n$, $n\geq 1$. Then, $(u_n,x_n)\in \mathcal{S}$.

On the other hand, observe that  
    \begin{align}\label{eq:unif_bound_min_seq}
        \|u_n\|_{\R^m}^2 +\|\E[C(\cdot)x_n(\cdot)]-z\|^2_{\R^n} = 2J^s(u_n)\leq 2J^s(0)+M,
    \end{align}
    for some constant $M>0$. Consequently, there exists $u$ and a subsequence $\{u_{n_k}\}_{k\geq 1}\in \R^m$ such that $u_{n_k}\to u$ in $\R^m$, as $k\to\infty$. Moreover, since $B$ is bounded we get
    \begin{align*}
        \|B(\cdot)u_{n_k}-B(\cdot)u\|_{L^2(\Omega;\R^n)}\leq \|B\|\|u_{n_k}-u\|_{\R^m}\to 0,
    \end{align*}
as $k\to\infty$. In particular, this implies that $Ax_{n_k}=Bu_{n_k}\to Bu$ in $L^2(\Omega;\R^n)$. Let us prove that there exist $x\in L^2(\Omega;\R^n)$ such that $(x_{n_k},u_{n_k})\rightharpoonup (x,u)\in \mathcal{S}$. Since $Ax_{n_k}\to Bu$ as $k\to\infty$, then the sequence $\{A_{n_k}\}_{k\geq 1}$ is bounded. Using the inequality \eqref{eq:ineq_statio_1} and estimation $\eqref{eq:unif_bound_min_seq}$ we deduce that the sequence $\{x_{n_k}\}_{k\geq 1}$ is bounded, and there exist $x\in L^2(\Omega;\R^n)$ such that $x_{n_k}\rightharpoonup x$, as $k\to\infty$. Therefore, $(x_{n_k},u_{n_k})\rightharpoonup (x,u)$. Due to the fact that $\mathcal{S}$ is closed and convex, it is weakly convex, and we have that $(x,u)\in \mathcal{S}$, i.e., $Ax=Bu$. From the strong convergence of $\{Bu_{n_k}\}_{k\geq 1}$ in $L^2(\Omega;\R^n)$, we also deduce that $Ax_k\to Ax$ in $L^2(\Omega;\R^n)$, as $k\to\infty$.

Observe that \eqref{eq:unif_bound_min_seq} also implies that there exist some $h\in \R^n$ such that $\E[Cx_{n_k}]\to h$ in $\R^n$, as $k\to\infty$. We claim that $h=\E[Cx]$. Indeed, let us assume that $h\neq \E[Cx]$, then for $N>0$ there exists $\delta>0$ such that $\|\E[Cx_{n_k}]-\E[Cx]\|^2_{\R^n}>\delta$ for all $n_k>N$. Now, using the strong convexity of the square Euclidean norm \cite[Section 2.3]{MR3310025}, i.e., for all $v,\,w\in \R^n$ and $\lambda\in [0,1]$,
\begin{align*}
    \|\lambda v+(1-\lambda)w\|^2_{\R^n}\leq \lambda\|v\|^2_{\R^n}+(1-\lambda)\|w\|^2_{\R^n} - \lambda(1-\lambda)\|v-w\|_{\R^n}^2,
\end{align*}
we have
\begin{align*}
    J^s\left(\frac{u+u_{n_k}}{2}\right)&=\frac{1}{2}\left(\left\|\frac{u+u_{n_k}}{2}\right\|^2_{\R^m}+\left\|\E\left[C(\cdot)\left(\frac{x(\cdot)+x_{n_k}(\cdot)}{2}\right)-z\right]\right\|_{\R^n}^2\right)\\
    &\leq \frac{1}{2}\left(\frac{1}{2}\|u\|^2_{\R^m}+\frac{1}{2}\|\E[C(\cdot)x(\cdot)]-z\|^2_{\R^n} + \frac{1}{2}\|u_{n_k}\|^2_{\R^m}+\frac{1}{2}\|\E[C(\cdot)x_{n_k}(\cdot)]-z\|^2_{\R^n} \right)\\
    &\hspace{8cm} - \frac{\|\E[C(\cdot)x(\cdot)]-\E[C(\cdot)x_{n_k}(\cdot)]\|^2_{\R^n}}{4}\\
    &=\frac{1}{2}\left(J^s(u)+J^s(u_{n_k})-\frac{\|\E[C(\cdot)x(\cdot)]-\E[C(\cdot)x_{n_k}(\cdot)]\|^2_{\R^n}}{2}\right)\\
    &\leq \frac{1}{2}\left(\liminf_{k\to \infty} \{J^s(u_{n_k})\}  + J^s(u_{n_k})-\frac{\|\E[C(\cdot)x(\cdot)]-\E[C(\cdot)x_{n_k}(\cdot)]\|^2_{\R^n}}{2}\right)\\
    &= \frac{1}{2}\left(\inf_{u\in\R^m}\{ J^s(u)\} + J^s(u_{n_k})-\frac{\|\E[C(\cdot)x(\cdot)]-\E[C(\cdot)x_{n_k}(\cdot)]\|^2_{\R^n}}{2}\right),
\end{align*}
where we have used the lower semicontinuity of the functional $J$ and \eqref{eq:minimization_sec_def}. Taking $k\to\infty$ and using that $\|\E[Cx_{n_k}]-\E[Cx]\|^2_{\R^n}>\delta$ for all $n_k>N$, we deduce that 
\begin{align*}
    \lim_{n\to \infty} J^s\left(\frac{u+u_{n_k}}{2}\right)< \inf_{u\in \R^m} J^s(u),
\end{align*}
contradicting the definition of infimum. Therefore $\E[Cx_{n_k}]\to \E[Cx]$ in $\R^n$ as $k\to \infty$. Then, using the stationary inequality \eqref{eq:ineq_statio_1} we have
\begin{align*}
    \|x_{n_k}-x\|_{L^2(\Omega;\R^n)}\leq K(\|A(\cdot)x_{n_k}(\cdot)-A(\cdot)x(\cdot)\|_{L^2(\omega;\R^n)}+\|\E[C(\cdot)x_{n_k}(\cdot)]-\E[C(\cdot)x(\cdot)]\|_{\R^m}).
\end{align*}
Thus, we have that $x_{n_k}\to x$ in $L^2(\Omega;\R^n)$, as $k\to\infty$. Finally, using again lower semicontinuity of $J^s$ we have that 
\begin{align*}
    J(u)\leq \liminf_{k\to\infty} J(u_{n_k})= \inf_{u\in \R^m} J^s(u).
\end{align*}
Concluding that $u$ minimize $J^s$. The uniqueness follows from the strict convexity.
\end{proof}

Since under Assumption \ref{A1}, the control has already been uniquely determined, we conclude this section with the characterization of the control.  By applying again the Lagrangian approach, we have that the optimal control of the stationary system denoted by $u^s$ can be characterized via the variational relation.
\begin{align}\label{eq:charac_statio_oc}
(u^s+\E[B^*\varphi^s],v)_{L^2(\Omega)}=0,\quad \text{for every } v\in D.
\end{align}
where $D$ is the set defined as 
\begin{align*}
    D:=\left\{ v\in \R^m\,:\, B(\omega)v \in Rank(A(\omega)) \text{ a.e. } \omega \in \Omega\right\},
\end{align*}
and $\varphi^s$ is the solution of 
\begin{align}\label{eq:stationary_adjoint_equation}
    A^*(\omega) v= C^*(\omega)\left(\E[C(\cdot)x^s(\cdot)]-z\right).
\end{align}
where $x^s$ denotes the solution of the equation \eqref{eq:statio_equation} with control given by \eqref{eq:charac_statio_oc}.

% \begin{remark}[Uniqueness of the stationary optimal control]\label{remark:uniq_statio_control} Observe that under assumption \ref{A1}, we can guarantee the uniqueness of a state associated with the optima control $u^s$. Indeed, assume that there exist two optimal states $x_1$ and $x_2$ associated with the optimal control $u^s$, i.e., $Ax_1=Bu^s=Ax_2$. On the other hand, we have that $\E[Cx_1]=\E[Cx_2]$ in $\R^n$. Indeed, if we suppose that $\E[Cx_1]\neq\E[Cx_2]$ then
% \begin{align*}
%     J^s(u)=J^s(\frac{u+u}{2})=\frac{1}{2}\left(\|u\|^2_{\R^m}+\left\|\E\left[C(\cdot)\left(\frac{x_1+x_2}{2}\right)\right]-z\right\|^2_{\R^{n}}\right)< \frac{J^s(u)+J^s(u)}{2}=J^s(u),
% \end{align*}
% which is a contradiction due to the strict convexity of $J^s$. Therefore, applying \eqref{eq:ineq_statio_1} we have
% \begin{align*}
%     \alpha_C\|x_1-x_2\|_{L^2(\Omega;\R^n)}\leq\|Ax_1-Ax_2\|_{L^2(\Omega;\R^n)}+|\E[Cx_1]-\E[Cx_2]|_{\R^n}=0.
% \end{align*}    
% \end{remark}

%%%%%%%%%%%%%%%%%%%%%%%%%
\section{Integral turnpike property}\label{sec:3}

In this section, the proof of Theorem \ref{Th:integral_turnpike} will be provided. First of all, let us recall the optimality systems for the time-dependent and the stationary problems. Let $(u^T,x^T)$ and $(u^s,x^s)$ the optimal pairs of optimal control problems \eqref{eq:evol_funtional} and \eqref{eq:statio_functional}, respectively. Then, from the previous section, there exists a pair $(\varphi^T,\varphi^s)$ such that
\begin{align}\label{charact-controls}
\begin{cases}
u^T(t)=-\E[B^* \varphi^T(t,\cdot)], \quad t\in(0,T),\\
(u^s+\E[B^*\varphi^s],v)_{L^2(\Omega)}=0,\quad \text{for every } v\in D,
\end{cases}
\end{align}
and we have the following optimality systems for a.e. $\omega\in\Omega$
\begin{align}\label{opt-evol-syst}
\begin{cases}
x_t^T(t,\omega)+A(\omega)x^T(t,\omega)=B(\omega)u^T(t),\quad t>0,\\
 -\varphi_t^T(t,\omega)+A^*(\omega)\varphi^T(t,\omega)= C^*(\omega)\left(\E[C(\cdot)x^T(t,\cdot)]-z\right),\quad t>0,\\
 x^T(0,\omega)=x_0(\omega),\\
  \varphi^T(T,\omega)=\varphi_T(\omega),
\end{cases}
\end{align}
and
\begin{align}\label{opt-stat-syst}
\begin{cases}
A(\omega)x^s(\omega)=B(\omega)u^s,\\
A^*(\omega) \varphi^s(\omega)= C^*(\omega)\left(\E[C(\cdot)x^s(\cdot)]-z\right).
\end{cases}
\end{align}

\begin{proof}[{\bf Proof of Theorem \ref{Th:integral_turnpike}}]
Let us observe that multiplying by $\varphi^T$ the first equation of \eqref{opt-evol-syst} driven with the optimal control $u^T$, and integrating by part over $(0,T)$ we obtain
\begin{multline*}
    (\varphi_T,x^T(T,\omega))_{\R^n}-(\varphi^T(0,\omega),x_0)_{\R^n}+\int_0^T(-\varphi_t^T(t,\omega)+A^*(\omega)\varphi^T(t,\omega),x^{T}(t,\omega))_{\R^n}dt\\   =\int_0^T(B(\omega)u^T(t),\varphi^T(t,\omega))_{\R^n}dt.
\end{multline*}
Thus, using the second equation of \eqref{opt-evol-syst} we deduce
\begin{multline*}
    (\varphi_T,x^T(T,\omega))_{\R^n}-(\varphi^T(0,\omega),x_0)_{\R^n}+\int_0^T(\E[C(\cdot)x^T(t,\cdot)]-z,C(\omega)x^{T}(t,\omega)-z)_{\R^n}dt\\
    +\int_0^T(\E[C(\cdot)x^T(t,\cdot)]-z,z)_{\R^n}dt
    =\int_0^T(B(\omega)u^T(t),\varphi^T(t,\omega))_{\R^n}dt.
\end{multline*}
Now, taking expectation and using the control characterization in \eqref{charact-controls}, we get
\begin{multline}\label{eq:relation_x_phi}
    \int_{0}^T\left(\|\E[C(\cdot)x^{T}(t,\cdot)-z]\|^2_{\R^n}+\|u^T(t)\|_{\R^m}^2 \right)dt =  (\varphi^T(0,\cdot),x_0(\cdot))_{L^2(\Omega;\R^n)}-(\varphi_T(\cdot),x(T,\cdot))_{L^2(\Omega;\R^n}\\-\int_{0}^T(\E[C(\cdot)x^T(t,\cdot)]-z,z)_{\R^n}dt.
\end{multline}
By applying Cauchy-Schwarz and Young inequalities at the right-hand side of \eqref{eq:relation_x_phi}, for every $\varepsilon_1,\,\varepsilon_2,\,\varepsilon_3>0$, we have
\begin{multline*}
    (\varphi^T(0,\cdot),x_0(\cdot))_{L^2(\Omega;\R^n)}+(\varphi_T(\cdot),x^T(T,\cdot))_{L^2(\Omega;\R^n)}+\int_{0}^T(\E[C(\cdot)x^T(t,\cdot)]-z,z)_{\R^n}dt\\\leq \frac{\varepsilon_1\|\varphi^T(0,\cdot)\|^2_{L^2(\Omega;\R^n)}}{2}+\frac{\|x_0(\cdot)\|^2_{L^2(\Omega;,\R^n)}}{2\varepsilon_1}+\frac{\varepsilon_2\|x^T(T,\cdot)\|^2_{L^2(\Omega;\R^n)}}{2}+\frac{\|\varphi_T(\cdot)\|^2_{L^2(\Omega;,\R^n)}}{2\varepsilon_2} \\+ \frac{\varepsilon_3}{2}\int_0^T\|\E[C(\cdot)x^{T}(t,\cdot)-z]\|^2_{\R^n}dt +T\frac{\|z\|_{\R^n}^2}{2\varepsilon_3}.
\end{multline*}
Thus, using Lemmas \ref{lem:inequality_state} and \ref{lem:inequality_adjoint}, we deduce
\begin{multline}\label{eq:estimation_initial_data}
    (\varphi^T(0,\cdot),x_0(\cdot))_{L^2(\Omega;\R^n)}+(\varphi_T(\cdot),x^T(T,\cdot))_{L^2(\Omega;\R^n)}+\int_{0}^T(\E[C(\cdot)x^T(t,\cdot)]-z,z)_{\R^n}dt\\
    \leq 
    \left(\frac{K_2\varepsilon_1}{2}+\frac{K_1\varepsilon_2}{2}+\frac{\varepsilon_3}{2}\right)\int_0^T\|\E[C(\cdot)x^{T}(t,\cdot)-z]\|^2_{\R^n}dt+\left( \frac{K_2\varepsilon_1}{2}+\frac{K_1\varepsilon_2}{2}\right)\int_0^T\|u^T(t)\|^2_{\R^m} dt \\+ \left(\frac{1}{2\varepsilon_3}\right)T\|z\|_{\R^n}^2+\frac{\|x_0(\cdot)\|^2_{L^2(\Omega;\R^n)}}{2\varepsilon_1}+\frac{\|\varphi_T(\cdot)\|^2_{L^2(\Omega;\R^n)}}{2\varepsilon_2}.
\end{multline}
Substituting estimation \eqref{eq:estimation_initial_data} into estimation \eqref{eq:relation_x_phi} yields
\begin{multline*}
     \left(1-\frac{K_2\varepsilon_1}{2}-\frac{K_1\varepsilon_2}{2}-\frac{\varepsilon_3}{2}\right)\int_0^T\|\E[C(\cdot)x^{T}(t,\cdot)-z]\|^2_{\R^n}dt+\left(1- \frac{K_2\varepsilon_1}{2}-\frac{K_1\varepsilon_2}{2}\right)\int_0^T\|u^T(t)\|^2_{\R^m} dt\\ \leq \left(\frac{K_1\varepsilon_2}{2}+\frac{1}{2\varepsilon_3}\right)T\|z\|_{\R^n}^2+\frac{\|x_0(\cdot)\|^2_{L^2(\Omega;\R^n)}}{2\varepsilon_1}+\frac{\|\varphi_T(\cdot)\|^2_{L^2(\Omega;\R^n)}}{2\varepsilon_2}.
\end{multline*}
Choosing $\varepsilon_1,\,\varepsilon_2,\,\varepsilon_3$ in order to satisfy $2>K_2\varepsilon_1+K_1\varepsilon_2+\varepsilon_3$ we deduce the existence of a constant  $K>0$ independent of $T>1$ such that 
\begin{align}\label{eq:estimation_CX_U}
   \int_0^T\|\E[C(\cdot)x^{T}(t,\cdot)-z]\|^2_{\R^n}dt+\int_0^T\|u^T(t)\|^2_{\R^m} dt\leq K T.
\end{align}
Therefore, due to Lemmas \ref{lem:inequality_state} and \ref{lem:inequality_adjoint} we conclude that
\begin{align}\label{eq:estimation_norm_XT_phiT}
    \|x^T(t,\cdot)\|^2_{L^2(\Omega;\R^n)}\leq KT,\quad\|\varphi^T(t,\cdot)\|^2_{L^2(\Omega;,\R^n)}\leq KT,
\end{align}
for every $t\in[0,T]$. 

Now, the following two terms
\begin{align*}
    \frac{1}{T}\int_0^T x^T(t,\cdot) dt,\quad  \frac{1}{T}\int_0^T u^T(t) dt,
\end{align*}
are uniformly bounded with respect to $T$ in $L^2(\Omega;\R^n)$ and $\R^m$, respectively. Indeed, integrating the first equation of \eqref{opt-evol-syst} over $(0,T)$, multiplying by $1/T$ and taking $L^2(\Omega; \R^n)$ norm, we have
\begin{align}
\label{eq:estimation_norm_Ax}
    \nonumber\left\|\frac{1}{T}\int_0^TA(\cdot)x^{T}(t,\cdot)dt\right\|_{L^2(\Omega;\R^n)}&\leq \left\|\frac{1}{T}\int_0^T B(\cdot)u^T(t)dt\right\|_{L^2(\Omega;\R^n)}+\left\|\frac{x^T(T,\cdot)-x_0(\cdot)}{T}\right\|_{L^2(\Omega;\R^n)}\\
    &\leq  \frac{\|B\|}{T^{1/2}}\left(\int_0^T \|u^T(t)\|_{\R^m}^2dt\right)^{1/2}+\frac{\|x^T(T,\cdot)\|_{L^2(\Omega;\R^n)}+\|x_0(\cdot)\|_{L^2(\Omega;\R^n)}}{T}\\
    &\nonumber\leq  K^{1/2}\|B\|_{\infty}+\frac{K^{1/2}}{T^{1/2}}+\frac{\|x_0(\cdot)\|_{L^2(\Omega;\R^n)}}{T},
\end{align}
where we have applied Holder's inequality and estimations \eqref{eq:estimation_CX_U} and \eqref{eq:estimation_norm_XT_phiT}. Thus, $\frac{1}{T}\int_0^TA(\cdot)x^{T}(t,\cdot)dt$ is uniformly bounded with respect to $T>T_{min}$. Then, using the stationary inequality \eqref{eq:ineq_statio_1}, the following is obtained
\begin{align}\label{eq:uniform_estimation_x}
   \nonumber\alpha_C \left\|\frac{1}{T}\int_0^Tx^{T}(t,\cdot)dt\right\|_{L^2(\Omega;\R^n)}^2&\leq \left\|\frac{1}{T}\int_0^TA(\cdot)x^{T}(t,\cdot)dt\right\|^2_{L^2(\Omega;\R^n)}+\left\|\frac{1}{T}\int_0^T\E[Cx^{T}(t,\cdot)]dt\right\|_{\R^n}^2\\
    &\leq  \left\|\frac{1}{T}\int_0^TA(\cdot)x^{T}(t,\cdot)dt\right\|^2_{L^2(\Omega;\R^n)}+\frac{1}{T^2}\int_0^T\|\E[Cx^{T}(t,\cdot)]\|^2_{\R^n} dt.
\end{align}
Due to the estimates \eqref{eq:estimation_CX_U} and \eqref{eq:estimation_norm_Ax} we can conclude that $\frac{1}{T}\int_0^Tx^{T}(t,\cdot)dt$ is uniform bounded in $L^2(\Omega;\R^n)$ with respect to $T>1$. On the other hand, let us observe that
\begin{align*}
    \left\|\frac{1}{T}\int_0^Tu^T(t)dt\right\|_{\R^m}\leq \frac{1}{T^{1/2}}\left(\int_0^T\|u^T(t)\|^2dt\right)^{1/2}\leq \sqrt K.
\end{align*}
And the proof of the claim is finished.

Let us introduce the new variables $y=x^T-x^s$ and $\phi=\varphi^T-\varphi^s$ that solve the system
\begin{align}\label{eq:system_equation_y_phi}
    \begin{cases}
        y_t(t,\omega)+A(\omega)y(t,\omega)=B(\omega)(u^T(t)-u^s),\quad &t\in(0,T),\\
        -\phi_t(t,\omega)+A^*(\omega)\phi(t,\omega) = C^*(\omega)\E[C(\cdot)y(t,\cdot)],&t\in(0,T),\\
        y(0,\omega)=x_0(\omega)-x^s(\omega),\quad \phi(0,\omega)=\varphi_T-\varphi^s(\omega).
    \end{cases}
\end{align}
Then, multiplying the first two equations of \eqref{eq:system_equation_y_phi} by $\phi$ and $y $, respectivley, integrating over $(0,T)$ and $\Omega$,  we obtain
\begin{multline*}
    \int_0^T\|\E[C(\cdot)y(t,\cdot)]\|^2_{\R^n}+\|u^T(t)-u^s\|^2_{\R^m}dt=(\phi(0,\cdot),y(0,\cdot))_{L^2(\Omega;\R^n)}-(\phi(T,\cdot),y(T,\cdot))_{L^2(\Omega;\R^n)}\\
    -\int_0^T\left(u^T(t)-u^s,u^s+\E[B^*(\cdot)\varphi^s(\cdot)]\right)_{\R^m}dt.
\end{multline*}
Applying Cauchy Schwarz's and Young's inequality at the right-hand-side we get
\begin{multline}\label{eq:estima_cy_u_us}
\frac{1}{2}\int_0^T\|\E[C(\cdot)y(t,\cdot)]\|^2_{\R^n}+\|u^T(t)-u^s\|^2_{\R^m}dt\leq(\phi(0,\cdot),y(0,\cdot))_{L^2(\Omega;\R^n)}\\
-(\phi(T,\cdot),y(T,\cdot))_{L^2(\Omega;\R^n)}
    +\frac{T}{2}\|u^s+\E[B^*(\cdot)\varphi^s(\cdot)]\|_{\R^m}^2.
\end{multline}
Now, using estimations \eqref{eq:estimation_norm_XT_phiT}, the following is obtained
\begin{multline}\label{eq:estima_phi0_y0}
(\phi(0,\cdot),y(0,\cdot))_{L^2(\Omega;\R^n)}+(\phi(T,\cdot),y(T,\cdot))_{L^2(\Omega;\R^n)}\\
\leq \|\phi(0,\cdot)\|_{L^2(\Omega;\R^n)}\|y(0,\cdot)\|_{L^2(\Omega;\R^n)} + \|\phi(T,\cdot)\|_{L^2(\Omega;\R^n)}\|y(T,\cdot)\|_{L^2(\Omega;\R^n)}\\
\leq ((KT)^{1/2}+2\|\varphi^s(\cdot)\|_{L^2(\Omega;\R^n)})(\|x_0(\cdot)-x^s(\cdot)\|_{L^2(\Omega;\R^n)}) \\+ (KT)
^{1/2}+2\|x^s(\cdot)\|_{L^2(\Omega;\R^n)})\|\varphi_T(\cdot)-\varphi^s(\cdot)\|_{L^2(\Omega;\R^n)}.
\end{multline}
Thus, due to \eqref{eq:estima_cy_u_us} and \eqref{eq:estima_phi0_y0}, there exist $K>0$ such that
\begin{align*}
    \frac{1}{T}\int_0^T\|\E[C(\cdot)y(t,\cdot)]\|^2_{\R^n}+\|u^T(t)-u^s\|^2_{\R^m}dt\leq K.
\end{align*}
In particular, we conclude that
\begin{align}
\label{int_turn_con}
    \left\|\frac{1}{T}\int_0^T\E[C(\cdot)y(t,\cdot)]dt\right\|^2_{\R^n}+\left\|\frac{1}{T}\int_0^T(u^T(t)-u^s)dt\right\|^2_{\R^m}\to0,
\end{align}
as $T\to \infty$. Concluding the integral turnpike property for the control $u^T$. % By \Cref{Th:uniq_statio_control} we know that $A(\omega)x^s(\omega)=B(\omega)u^s$ 
Integrating the first equation in \eqref{opt-evol-syst} once more, subtracting the $A(\omega)x^s(\omega)$ term from it, and taking $L^2(\Omega;\R^n)$ norm, we have
\begin{multline}
    \label{int_turn_y}\left\|A(\cdot)\left(\frac{1}{T}\int_{0}^T x^T(t,\cdot)dt -x^s(\cdot) \right) \right\|_{L^2(\Omega;\R^n)}^2\leq 2\left\|B(\cdot)\left(\frac{1}{T}\int_{0}^T(u^T(t)-u^s)dt \right)\right\|_{L^2(\Omega;\R^n)}^2\\
    +2\left\|\frac{x^T(T,\cdot)-x_0(\cdot)}{T}\right\|_{L^2(\Omega;\R^n)}^2 \to 0,
\end{multline}
as $T\to \infty$, due to \eqref{eq:estimation_norm_XT_phiT} and the integral turnpike property for the control \eqref{int_turn_con}. Finally, using the stationary inequality \eqref{eq:ineq_statio_1} we conclude that 
\begin{multline*}
    \left\|\frac{1}{T}\int_{0}^T x^T(t,\cdot)dt -x^s(\cdot)  \right\|_{L^2(\Omega;\R^n)}^2\leq 2 \left\|A(\cdot)\left(\frac{1}{T}\int_{0}^T x^T(t,\cdot)dt -x^s(\cdot) \right) \right\|_{L^2(\Omega;\R^n)}^2\\
    + 2\left\|\frac{1}{T}\int_{0}^T \E[C(x^T(t,\cdot) -x^s)]dt \right\|_{\R^n}^2.
\end{multline*}
Taking into account \eqref{int_turn_con} and \eqref{int_turn_y}, we conclude the integral turnpike for $x^{T}$ and the proof is finished.
\end{proof}

As a consequence of the previous result, the characterization of the stationary optimal control can improved, which is given by the next corollary.
\begin{corollary}\label{cor:characterization_us}
    Let us suppose that Assumptions \ref{A1} and \ref{A2} hold. Then, the stationary optimal control can be characterized by 
    \begin{align*}
        u^s=-\E[B^*(\cdot)\varphi^s(\cdot)].
    \end{align*}
\end{corollary}
\begin{proof}
    By Theorem \ref{Th:integral_turnpike} we know that the temporal average of $x^T$ and $u^T$ converge to the stationary variables $x^s$ and $u^s$ in $L^2(\Omega; \R^n)$ and $\R^m$, respectively. Then, applying the same steps used to bound $\frac{1}{T}\int_0^T x^T(t,\cdot)dt$ in $L^2(\Omega;\R^n)$ (see \eqref{eq:estimation_norm_Ax}-\eqref{eq:uniform_estimation_x}), we can conclude that $\frac{1}{T}\int_0^T\varphi^T(t,\omega)dt$ is bounded in $L^2(\Omega;\R^n)$. Thus, there exists $\varphi_l$ in $L^2(\Omega;\R^n)$ such that
    \begin{align*}
        \frac{1}{T}\int_0^T\varphi^T(t,\omega)dt\rightharpoonup \varphi_l, \quad \text{in }L^2(\Omega;\R^n).
    \end{align*}    
    On the other hand, taking the time average of the second equation of \eqref{eq:system_equation_y_phi} and taking $L^2(\Omega;\R^n)$-norm, we get
\begin{align*}
    \left\|A^*(\cdot)\frac{1}{T}\int_0^T\phi(t,\cdot)dt\right\|^2_{L^2(\Omega;\R^n)}\leq \frac{1}{T^2}\int_0^T\|\E[C(\cdot)y(t,\cdot)]\|^2_{\R^n}dt+\frac{\|\phi(T,\cdot)\|_{L^2(\Omega;\R^n)}^2+\|\phi(0,\cdot)\|_{L^2(\Omega;\R^n)}^2}{T}\to0,
\end{align*}
as $T\to \infty$, where $\phi=\varphi^T-\varphi^s$. Now, taking the time average and the $L^2(\Omega;\R^n)$-norm of the second equation of \eqref{opt-evol-syst}, and subtracting the term $A(\omega)\varphi^s(\omega)$, we have
\begin{multline*}
     \left\|\frac{1}{T}\int_0^TC^*(\cdot)\E[C(\cdot)x^T(t,\cdot)]dt-A^*(\cdot)\varphi^s(\cdot)\right\|^2_{L^2(\Omega;\R^n)}\leq  \left\|A^*(\cdot)\frac{1}{T}\int_0^T\phi(t,\cdot)dt\right\|^2_{L^2(\Omega;\R^n)}\\+\frac{\|\varphi_T(\cdot)\|_{L^2(\Omega;\R^n)}^2+\|\varphi(0,\cdot)\|_{L^2(\Omega;\R^n)}^2}{T}\to0,
\end{multline*}
as $T\to \infty$. Therefore, $\varphi_l=\varphi^s$ is a solution of the adjoint system \eqref{opt-stat-syst}. Finally, observe that for every $\zeta\in \R^m$ we have that
\begin{multline*}
    (u^s+\E[B^*(\cdot)\varphi^s(\cdot)],\zeta)_{\R^m}=\left(u^s-\frac{1}{T}\int_0^Tu^T(t)dt,\zeta\right)_{\R^m}\\
    +\left(\frac{1}{T}\int_0^Tu^T(t)-\E\left[B^*(\cdot)\frac{1}{T}\int_0^T\varphi^T(t,\cdot)dt\right],\zeta\right)_{\R^m}+\left(\E\left[B^*(\cdot)\frac{1}{T}\int_0^T\varphi^T(t,\cdot)dt\right]-\E\left[B^*(\cdot)\varphi(\cdot)\right],\zeta\right)_{\R^m}\\
    =\left(u^s-\frac{1}{T}\int_0^Tu^T(t)dt,\zeta\right)_{\R^m}+\left(\frac{1}{T}\int_0^T\varphi^T(t,\cdot)dt-\varphi(t,\cdot),B(\cdot)\zeta\right)_{L^2(\Omega;\R^n)}\to 0,
\end{multline*}
as $T\to\infty$ since $B(\cdot)\zeta \in L^2(\Omega;\R^n)$. Note that the second term after the first equation vanishes due to the characterization of evolutionary optimal control \eqref{eq:chara_evol_oc}. Taking $\zeta=u^s+\E[B^*(\cdot)\varphi^s(\cdot)]$ we conclude the result.
\end{proof}

\begin{remark}[Uniqueness of the stationary adjoint state]
Corollary \ref{cor:characterization_us} allows us to ensure the uniqueness of the stationary adjoint state $\varphi^s$, the solution of the second equation in \eqref{opt-stat-syst}. Indeed, assume that there exist two solutions $\varphi_1^s$ and $\varphi_2^s$ of \eqref{opt-stat-syst}. Then, we have that $A^*(\omega)\varphi_1^s(\omega) = A^*(\omega)\varphi_2^s(\omega)$. Furthermore, considering Corollary \ref{cor:characterization_us} and the uniqueness of the optimal control, we have $\E[B^*(\cdot)\varphi_1^s)(\cdot)] = \E[B^*(\cdot)\varphi_2^s(\cdot)]$. Finally, the uniqueness is established by applying the stationary inequality \eqref{eq:ineq_statio_2}.
\end{remark}

%%%%%%%%%%%%%%%%%%%%%%%%%%
\section{Exponential turnpike property}\label{sec:4}
Let us now give the proof of the second main result of our article, that is, Theorem \ref{Th:exponential_turnpike}.

\begin{proof}[{\bf Proof of Theorem \ref{Th:exponential_turnpike}}]
     As in the proof of Theorem \ref{Th:integral_turnpike} let $y=x^T-x^s$ and $\phi=\varphi^T-\varphi^s$ satisfying system \eqref{eq:system_equation_y_phi}. Applying Lemma \ref{lem:inequality_state} to $y$ and using the characterization of the stationary control given by Corollary \ref{cor:characterization_us}, there exists $K_1>0$ such that 
\begin{align}\label{eq:des_y_lemma1}
    \|y(t,\cdot)\|_{L^2(\Omega;\R^n)}^2\leq K_1 \int_0^T\Big(\|\E[B^*(\cdot)\phi(t,\cdot)]\|^2_{\R^m}+ \|\E[C(\cdot)y(t,\cdot)]\|^2_{\R^n}\Big)dt+\|y(0,\cdot)\|^2_{L^2(\Omega;\R^{n})},
\end{align}
    for every $t\in[0,T]$. In a similar way, applying Lemma \ref{lem:inequality_adjoint} to $\phi$ we deduce
    \begin{align}\label{eq:des_phi_lemma2}
        \|\phi(t,\cdot)\|_{L^2(\Omega;\R^n)}^2\leq K_1 \int_0^T\Big(\|\E[B^*(\cdot)\phi(t,\cdot)]\|^2_{\R^m}+ \|\E[C(\cdot)y(t,\cdot)]\|^2_{\R^n}\Big)dt+\|\phi(T,\cdot)\|^2_{L^2(\Omega;\R^{n})},
    \end{align}
    for every $t\in[0,T]$. Using \eqref{eq:estima_cy_u_us} and Corollary \ref{cor:characterization_us}, we have the following inequality
    \begin{multline*}
        \int_0^T\Big(\|\E[C(\cdot)y(t,\cdot)]\|^2_{\R^n}+\|\E[B^*(\cdot)\phi(t,\cdot)]\|^2_{\R^m}\Big)dt\leq \|\phi(0,\cdot)\|_{L^2(\Omega;\R^{n})}\|y(0,\cdot)\|_{L^2(\Omega;\R^{n})}\\
        +\|\phi(T,\cdot)\|_{L^2(\Omega;\R^{n})}\|y(T,\cdot)\|_{L^2(\Omega;\R^{n})}.
    \end{multline*}
    Thus, using Young's inequality for $\varepsilon_1,\,\varepsilon_2>0$ and estimations \eqref{eq:des_y_lemma1} and \eqref{eq:des_phi_lemma2} we deduce
      \begin{multline*}
      \int_0^T\Big(\|\E[C(\cdot)y(t,\cdot)]\|^2_{\R^n}+\|\E[B^*(\cdot)\phi(t,\cdot)]\|^2_{\R^m}\Big)dt\leq \|y(0,\cdot)\|_{L^2(\Omega;\R^{n})}^2\left(\frac{1}{2\varepsilon_1}+\frac{\varepsilon_2}{2}\right)\\+\|\phi(T,\cdot)\|_{L^2(\Omega;\R^{n})}^2\left(\frac{1}{2\varepsilon_2}+\frac{\varepsilon_1}{2}\right)
      +\left(\frac{\varepsilon_1K_2+\varepsilon_2K_1}{2}\right)\int_0^T\Big(\|\E[C(\cdot)y(t,\cdot)]\|^2_{\R^n}+\|\E[B^*(\cdot)\phi(t,\cdot)]\|^2_{\R^m}\Big)dt.
    \end{multline*}
    Therefore, taking $2>\varepsilon_1K_2+\varepsilon_2K_1$ and using inequalities \eqref{eq:des_y_lemma1} and \eqref{eq:des_phi_lemma2} ones again we conclude that there exists $K>0$ independent of  $T$ such that
    \begin{align}\label{eq:estimation_y_phi_in_X}
        \|y\|_{\mathcal{X}}^2  +\|\phi\|_{\mathcal{X}}^2\leq K(\|y(0,\cdot)\|_{L^2(\Omega;\R^{n})}^2+\|\phi(T,\cdot)\|_{L^2(\Omega;\R^{n})}^2),
    \end{align}
    where $\mathcal{X}$ denotes the space $C^0([0,T];L^2(\Omega;\R^n))$. On the other hand, system \eqref{eq:system_equation_y_phi} can be rewritten in the matrix form
    \begin{align*}
        \Lambda\begin{pmatrix} y\\\phi\end{pmatrix}:= \begin{pmatrix}
            -C^*(\omega)\E[C(\cdot)\,\cdot\, ] & -\frac{d}{dt} + A^*(\omega)\\
            0& E_T\\
            \frac{d}{dt} +A(\omega) & B(\omega)\E[B^*(\cdot)\,\cdot\, ]\\
            E_0& 0
        \end{pmatrix}\begin{pmatrix} y\\\phi\end{pmatrix}=\begin{pmatrix} 0\\ \phi(T,\omega)\\0 \\ y(0,\omega)\end{pmatrix}=:\mathcal{Y},
    \end{align*}
where $E_T(\phi):=\phi(T,\omega)$ and $E_0(y):=y(0,\omega)$, a.e. $\omega\in\Omega$. Let us observe that the linear operator $\Lambda^{-1}$ corresponds to the solution operator of the optimality system \eqref{eq:system_equation_y_phi}. By the uniqueness of $y$ and $\phi$, we deduce that the operator $\Lambda^{-1}:\mathcal{W} \to \mathcal{X}^2$ is well defined, where $\mathcal{W}:= (L^2(0,T;L^2(\Omega))\times L^2(\Omega;\R^n))^2$. Moreover, since \eqref{eq:estimation_y_phi_in_X} holds, we can ensure that $\Lambda^-1$ is uniformly bounded with respect $T$, that is,
\begin{align*}
    \|\Lambda^{-1}\|_{\mathcal{L}(\mathcal{W}, \mathcal{X}^2)}\leq K.
\end{align*}
Now, let us consider the new variable 
\begin{align*}
    \hat{y}(t,\omega):=y(t,\omega)f(t),\quad  \hat{\phi}(t,\omega):=\phi(t,\omega)f(t),\quad f(t)=\frac{1}{e^{-\delta t}+e^{-\delta(T-t)}},
\end{align*}
where $\delta>0$ will be chosen later on. Observe that the following system holds
\begin{align}\label{eq:system_variables_hat}
    \begin{cases}
\hat{y}_t(t,\omega)+A(\omega)\hat{y}(t,\omega)=-B(\omega)\E[B^*(\cdot)\hat{\phi}(t,\cdot)]+\delta h(t),\quad t\in (0,T),\\
-\hat{\phi}_t(t,\omega)+A^*\hat{\phi}(t,\omega)=C^*(\omega)\E[C(\cdot)\hat{y}(t,\cdot)]-\delta h(t),\quad t\in(0,T),\\
\hat{y}(0,\omega)=y(0,\omega)f(t),\quad \hat{\phi}(T,\omega)=\phi(T,\omega)f(T),
    \end{cases}
\end{align}
where $h(t):=(e^{-\delta t}-e^{-\delta(T-t)})f(t)$. Then using the definition of $\Lambda$ and the fact that $f(T)=f(0)<1$, equation \eqref{eq:system_variables_hat} can be rewritten as
\begin{align}\label{eq:Lambda_P_z}
   (\Lambda-\delta\mathcal{P})\mathcal{Z}:= \left[\Lambda-\delta\begin{pmatrix}
        h(t)&0\\
        0&0\\
        0&-h(t)\\
        0&0
    \end{pmatrix}\right]\begin{pmatrix}
        \hat{y}\\ \hat{\phi}
    \end{pmatrix}=f(0)\mathcal{Y}.
\end{align}
Observe that $\|P\|_{\mathcal{L}(\mathcal{X}^2,\mathcal{W})}\leq 1$. Then, applying $\Lambda^{-1}$ from the left of equation \eqref{eq:Lambda_P_z} we obtain 
\begin{align}\label{eq:to_invert}
    (I-\delta\Lambda^{-1}\mathcal{P})\mathcal{Z}=\Lambda^{-1}\mathcal{Y}.
\end{align}
Observe that for a $\delta>0$ small enough we can guaranty that $\gamma:=\delta\|\Lambda\|_{\mathcal{L}(\mathcal{W},\mathcal{X}^2)}<1$, where $\gamma>0$ is a constant independent of $T$ due to \eqref{eq:estimation_y_phi_in_X}. Therefore, since \cite[Theorem 2.14]{MR3184286}, there exists a unique solution $\mathcal{Z}$ of \eqref{eq:to_invert}. Moreover, we have the following estimation 
\begin{align*}
    \|(I-\delta\Lambda^{-1}\mathcal{P})^{-1}\|_{\mathcal{L}(\mathcal{X}^2)}\leq \frac{1}{1-\gamma}
\end{align*}
Namely, applying Therefore, by multiplying equation \eqref{eq:to_invert} by $(I-\delta \Lambda^{-1} \mathcal{P})^{-1}$ and taking the norm in $\mathcal{X}$, we have
\begin{align*}
    \|\hat{y}\|_{\mathcal{X}}+ \|\hat{\phi}\|_{\mathcal{X}}\leq \frac{\|\Lambda^{-1}\|_{\mathcal{L}(\mathcal{W},\mathcal{X}^2)}}{1-\gamma}f(0)(\|y(0,\cdot)\|_{L^2(\Omega;\R^{n})}^2+\|\phi(T,\cdot)\|_{L^2(\Omega;\R^{n})}^2).
\end{align*}
Returning to the original variables, we conclude that 
\begin{multline*}
    \|x^{T}(t,\cdot)-x^s(\cdot)\|_{L^2(\Omega;\R^n)}+   \|\varphi^{T}(t,\cdot)-\varphi^s(\cdot)\|_{L^2(\Omega;\R^n)}\\
    \leq K(\|x_0(\cdot)-x^s(\cdot)\|_{L^2(\Omega;\R^n)}+\|\varphi_T(\cdot)-\varphi^s(\cdot)\|_{L^2(\Omega;\R^n)})(e^{-\delta t}+e^{-\delta(T-t)}),
\end{multline*}
where $K>0$ is the same constant from \eqref{eq:estimation_y_phi_in_X}.
\end{proof}

%%%%%%%%%%%%%%%%%%%%%%%%%
\section{Numerical example}\label{sec:5}
This section is devoted to illustrate our theoretical results with a numerical example. Let us fix a time horizon $T=10$, and we introduce the random variables $\alpha,\,\beta\sim \text{Pois}(\lambda)$ with $\lambda=5$. We consider the evolutive optimal control problems \eqref{eq:evol_funtional}-\eqref{eq:evol_equation} and stationary \eqref{eq:statio_functional}-\eqref{eq:statio_equation} with matrices
\begin{align}\label{num-exam}
A(\omega)=\alpha(\omega)\begin{pmatrix}
        2&-5\\5&0.1
    \end{pmatrix},\quad B(\omega)=\beta(\omega)\begin{pmatrix}
        5\\7
    \end{pmatrix},\quad C=\begin{pmatrix}
        0&1\\1&0
    \end{pmatrix},\quad \text{and}\quad z=\begin{pmatrix}
        4\\4
    \end{pmatrix}.
\end{align}
Observe that with these matrices, Assumptions \ref{A1} and \ref{A2} holds. The numerical simulation of the equation \eqref{eq:evol_equation} uses the implicit Runga–Kutta method with $n=150$ time steps. The optimal control problem \eqref{eq:evol_funtional} is solved by the gradient descent method. The above is implemented through Gekko library \cite{pr6080106} in Python. The problem \eqref{eq:statio_functional}-\eqref{eq:statio_equation} is solved using Gekko, which this time uses the free solver APOPT. Additionally, to approximate the expectation, we consider several realizations of the random variables. The expectation of states and controls of both problems are shown in Figure \ref{fig1}.

\begin{figure}[h!]
\centering
\subfloat[Average of the optimal evolutive and stationary states. In light blue, different realizations of the states are illustrated to approximate the average.]{
\includegraphics[width=0.45\textwidth]{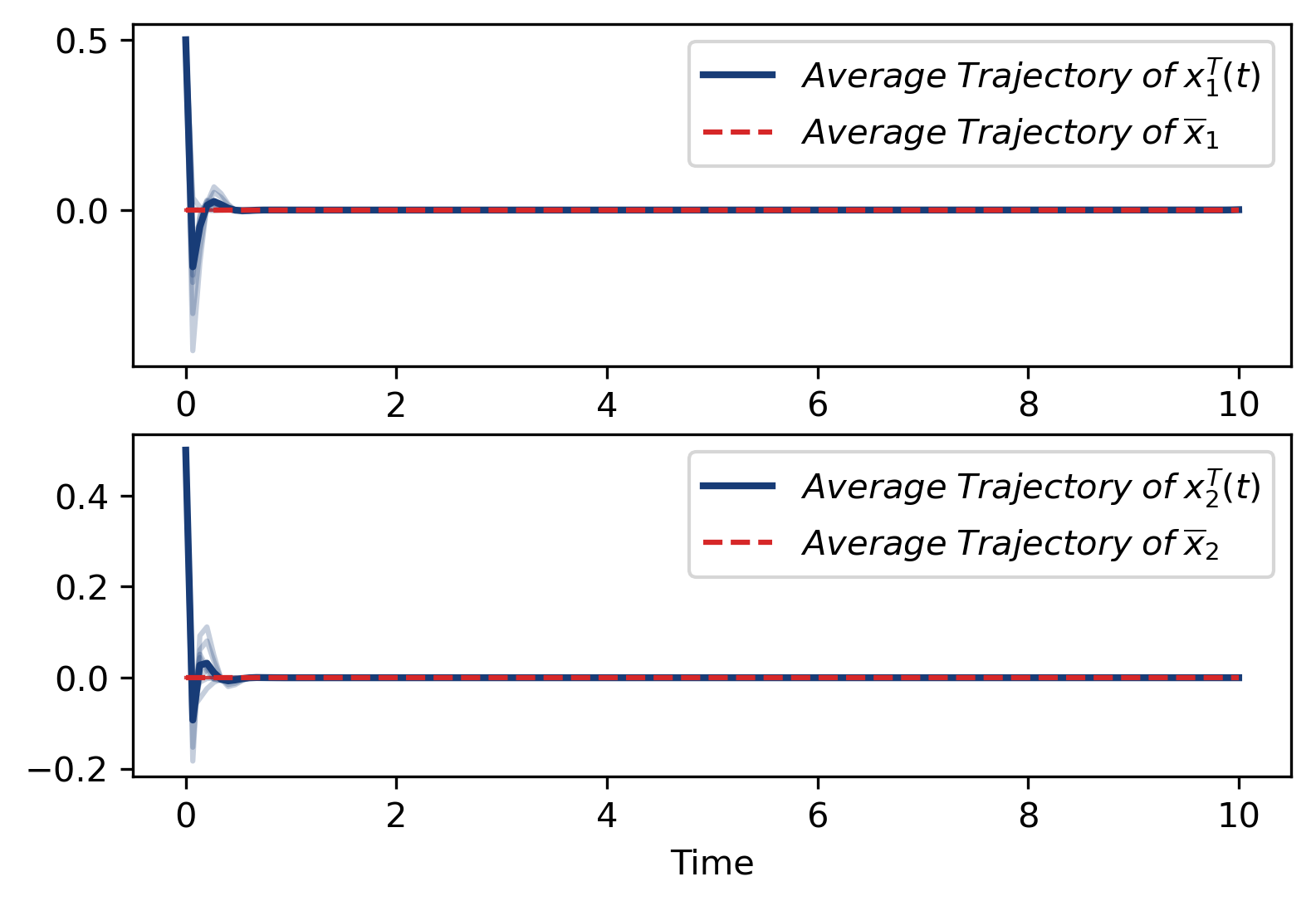}}
\subfloat[Optimal controls of the evolutive and stationary states. The second plot illustrates the turnpike property upper bound.]{
\includegraphics[width=0.45\textwidth]{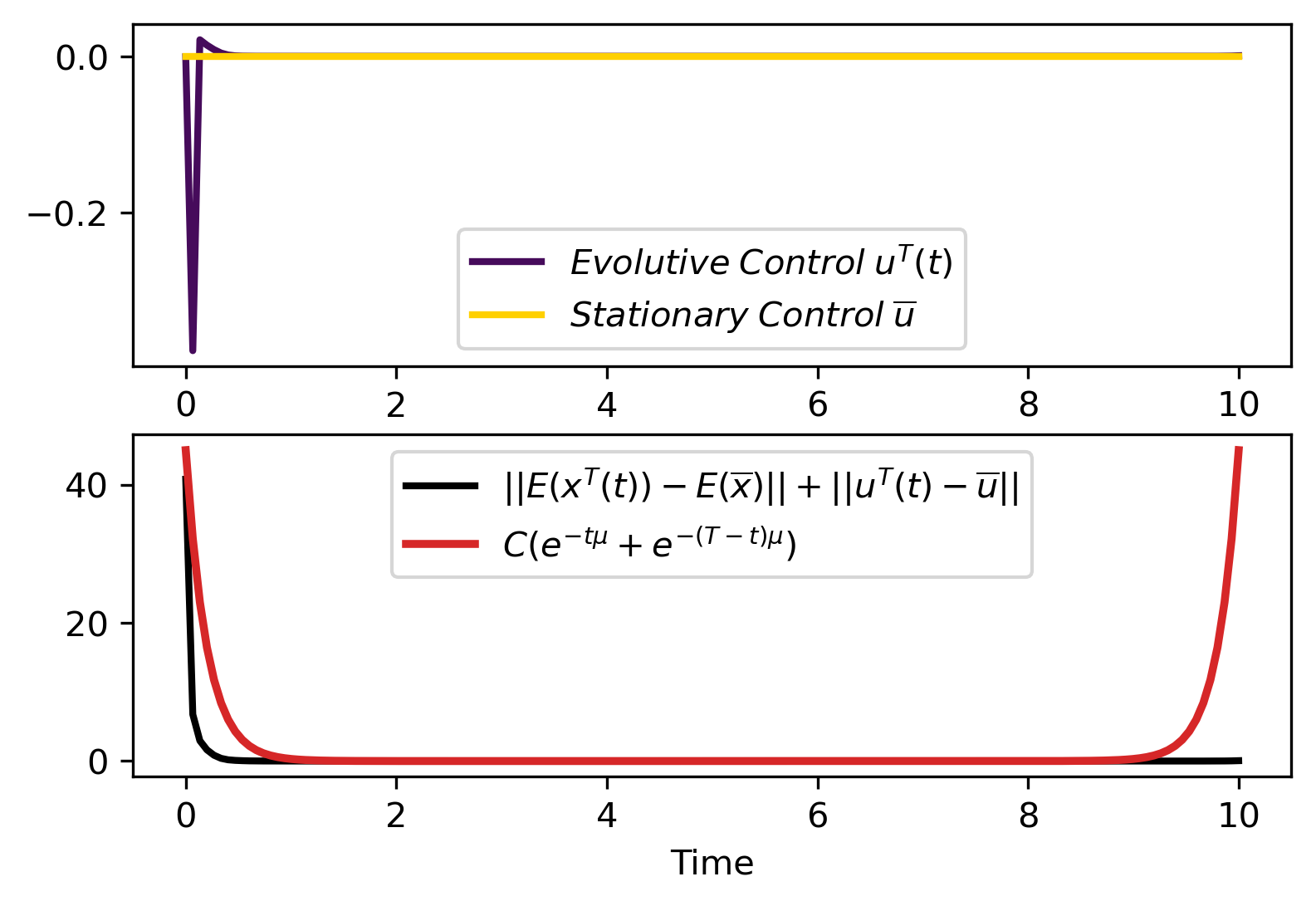}}\caption{Illustration of the turnpike property for optimal states and controls of \eqref{eq:evol_funtional} and \eqref{eq:statio_functional}.}\label{fig1}
\end{figure}

Let us note that the matrices given in \eqref{num-exam} satisfy the hypotheses of \cite{HLZ23}, and therefore, the turnpike property for the optimal control problem is considered there. The optimal states for the functional defined in \cite{HLZ23} are illustrated in Figure \ref{fig2},

\begin{figure}[h!]
    \centering
    \includegraphics[width=0.45\textwidth]{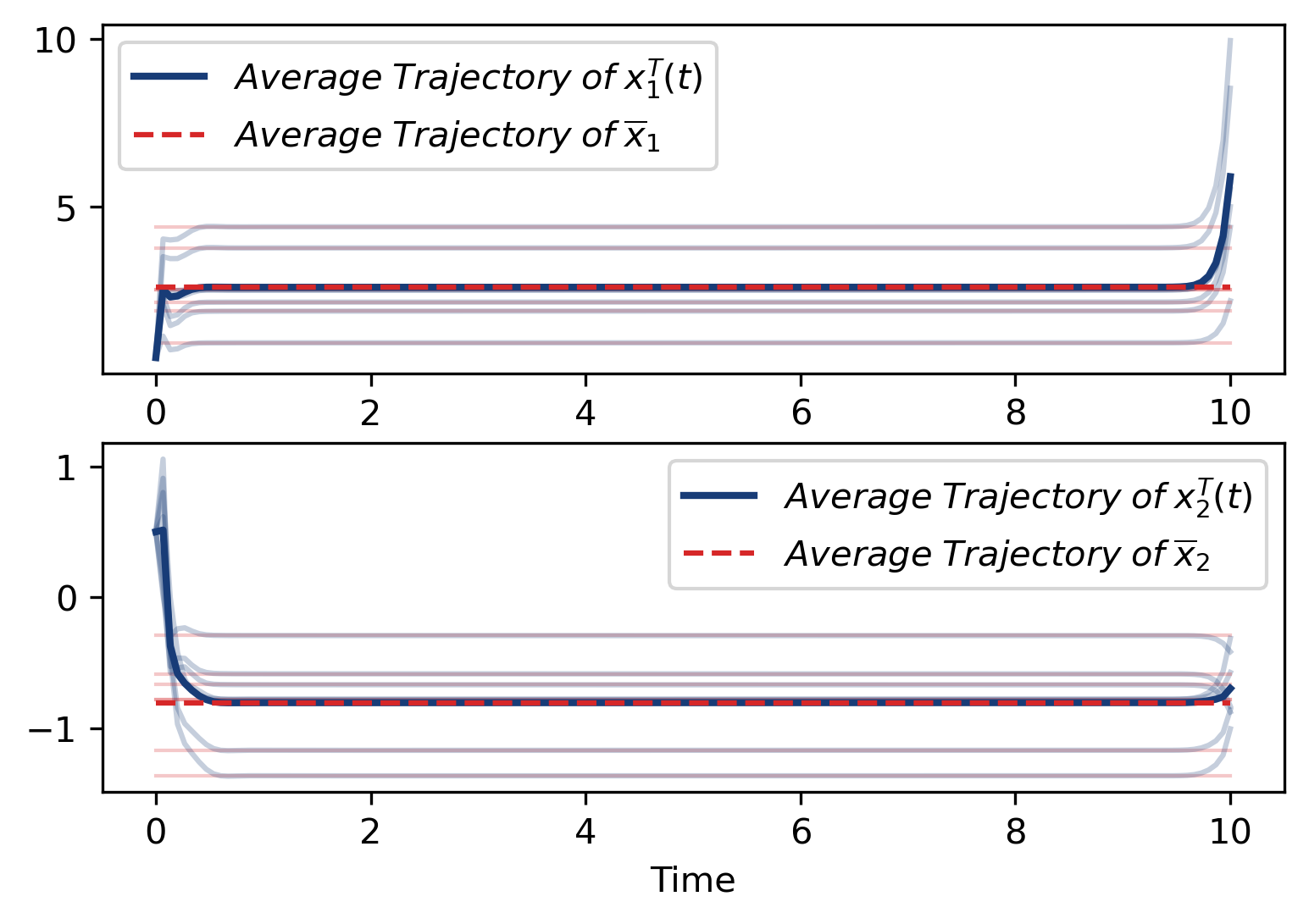}
    \caption{Optimal evolutive and stationary states of the functionals defined in \cite{HLZ23}. In light blue and red, we illustrate the evolutive and stationary optimal states, respectively, for different realizations.}
    \label{fig2}
\end{figure}

From Figure \ref{fig1} and \ref{fig2}, it is immediate to notice the big difference between the result in this article and in \cite{HLZ23}. In our case, we have only penalized the expectation of the state, unlike in \cite{HLZ23} where the cost functional considers the penalty of the whole state. We refer to the next section for a better understanding of this difference and how the hypotheses of our article and \cite{HLZ23} are related.

%%%%%%%%%%%%%%%%%%%%%%%%%
\section{Open problems and further commentaries}\label{sec:6}
In this article, the integral and exponential turnpike properties for parameter-dependent equations with averaged observations have been proved. The main hypotheses in order to obtain these results are Assumptions \ref{A1} and \ref{A2}. We called the $(A,C)$ detectability and $(A,B)$ stabilizability on average assumptions. For a better understanding of these concepts and their relationship to the existing literature, we present a non-exhaustive discussion on them.

\medbreak

$\bullet$  We start with an interpretation of Assumption \ref{A1} with respect to the uniform detectability condition, a concept which is given in the following definition.
%    Interpretation of Assumptions \ref{A1} and \ref{A2}, and its implication to the stabilization.
\begin{definition}
 A pair $(A,C)\in  C^0(\Omega,\mathcal{L}(\R^n))^2$ is uniformly detectable if there exists a constant $\alpha>0$ parameter-independent, such that for almost every $\omega\in \Omega$ there exists $K_C(\omega)$ such that for all $v\in\R^n$
   \begin{align}
       \label{unif_det}
    ((A(\omega)+K_C (\omega)C(\omega)) v, v)_{\R^n}\geq \alpha\|v\|_{\R^n}^2.
   \end{align}
  \end{definition}

In order to interpret our assumption, we shall use the following lemma. 
\begin{lemma}
    A pair $(A,C)\in (\mathcal{L}(\R^n))^2$ is not detectable if and only if for every $\alpha>0$  there there exists a  $v_\alpha \in \R^n$ such that
    \begin{align}
    \label{not_det_equi}
       (Av_\alpha,v_\alpha)_{\R^n}< \alpha\|v_\alpha\|_{\R^n}^2 \quad {\rm and} \quad Cv_\alpha =0. 
    \end{align}
\end{lemma}
\begin{proof}
    %The reverse implication is trivial. 
    The proof of the sufficient condition is a direct computation.

    In order to prove the necessary condition, we argue by contradiction, that is, there exists $\alpha>0$ such that for every $v$ either $(Av,v)_{\R^n}\geq \alpha\|v\|_{\R^n}^2$ or $Cv \not= 0$. 

    Define $K_C=\frac{\tilde \alpha}{\lambda^2} C^*$, where 
    \begin{align}
    \lambda=\inf_{v\notin \ker C}\frac{|Cv|}{|v|} >0, \label{sz-open}
    \end{align}
    and
    \begin{align*}
    \tilde \alpha > \alpha +\sup_{v\notin \ker A} (-\frac{(Av,v)}{|v|^2} ).
    \end{align*}
Let us observe that \eqref{sz-open} is not trivial in infinite dimension.

    Then it is easy to check that for such constructed $K_C$, for every $v\in\R^n$
    $$
    ((A+K_C C) v, v)_{\R^n}\geq \alpha\|v\|_{\R^n}^2,
    $$
    i.e. the pair $(A,C)$ is detectable, which contradicts the assumption. 
\end{proof}

Under a procedure analogous to the previous one, the following result is immediately obtained.
\begin{lemma}
\label{unif_det_lemma}
    A pair $(A,C)\in  C^0(\Omega,\mathcal{L}(\R^n))^2$ is not uniformly detectable on $\Omega$ if and only if  for every $\alpha>0$  there exists a subset $\Omega_1\subset \Omega$ of a positive measure such that for every $\omega\in \Omega_1$ there is $v_\alpha \in L^2(\Omega_1;\R^n)$ such that 
%  (\forall \alpha) (\exists \Omega_1\subseteq\Omega) |\Omega_1| >0 (\forall \omega \in \Omega_1) (\exists v_0) 
    \begin{align}
    \label{not_unifdet_equi}
      (A(\omega) v_\alpha(\omega),v_\alpha(\omega))_{\R^n}< \alpha\|v_\alpha(\omega)\|_{\R^n}^2 \quad {\rm and} \quad C(\omega) v_\alpha(\omega) =0. 
    \end{align}
\end{lemma}

Finally, we can prove the following, which gives us an interpretation of our hypothesis.
\begin{lemma}
\label{assum_interp}
   Suppose that Assumption \ref{A1} holds. Then the pair $(A,C)$ is uniformly detectable on $\Omega$. 
      \end{lemma}
\begin{proof}
Assume the contrary. Then, according to Lemma \ref{unif_det_lemma} for every $\alpha>0$, there exists a subset $\Omega_1$ of a positive measure such that for every $\omega\in \Omega_1$ there is $v_\alpha (\omega)$ such that relation \eqref{not_unifdet_equi} holds. 

Let us now define  $\tilde v(\omega)= v_\alpha (\omega) \chi_{\Omega_1}(\omega)$. Then $C(\omega) v(\omega)=0$ for every $\omega \in \Omega$. 
Consequently, for every $K_C\in C^0(\Omega,\mathcal{L}(\R^n))$
 \begin{align*}
       (A\tilde v,\tilde v)_{L^2(\Omega;\R^n)}  + |E(K_C Cv)|^2  =  (A\tilde v,\tilde v)_{L^2(\Omega_1;\R^n)} < \alpha\|v_\alpha(\omega)\|_{L^2(\Omega_1;\R^n)}^2=\alpha\|v_\alpha(\omega)\|_{L^2(\Omega;\R^n)}^2,
    \end{align*}
    which contradicts Assumption \ref{A1}. 
  \end{proof}

\medbreak

$\bullet$ The same procedure holds for the Assumption \ref{A2} but now with the uniformly stabilizability condition on $\Omega$. 
    
\medbreak

$\bullet$ Now, in \cite{HLZ23}, the turnpike property was established under similar kinds of assumptions used in this paper. Actually, one takes the following form.
    \begin{assumption}\label{A0}
     There exists a  
feedback operator $K_C\in C^0(\Omega,\mathcal{L}(\R^n))$ and a positive constant $\alpha_C>0$ such that
    \begin{align}
    \label{A0_ineq}
    (Av+ K_C Cv,v)_{L^2(\Omega;\R^n)}\geq \alpha_C\|v\|_{L^2(\Omega;\R^n)}^2,\quad \text{for every }v\in L^2(\Omega;\R^n).
\end{align}
\end{assumption}

The following lemma provides the interpretation of this assumption. 
\begin{lemma}
    Assumption \ref{A0} is equivalent to the statement that the pair $(A(\omega), C(\omega))$ is uniformly detectable, i.e. it satisfies the detectability assumption \eqref{unif_det} for almost every $\omega \in \Omega$ with the parameter independent constant $\alpha$.
\end{lemma}
\begin{proof}
   Let us first suppose \ref{A0}  holds for some $K_C\in C^0(\Omega,\mathcal{L}(\R^n))$ and $\alpha_C>0$. Arguing by contradiction, suppose the pair  $(A(\omega), C(\omega))$ is not uniformly detectable. Then  there exists a parameter subset $\Omega_1 \subseteq \Omega$ of a positive measure such that for every $\omega \in \Omega_1$ there exist a vector $v_0(\omega)$ such that
     \begin{align}
     \label{contr_det}
    ((A(\omega)+ K_C(\omega) C(\omega))v_0(\omega),v_0(\omega))_{\R^n)} <  \alpha_C\|v_0(\omega)\|_{\R^n}^2,
 \end{align}
 where operators $K_C(\omega)$ and the constant $\alpha$ are those appearing in \eqref{A0_ineq}.
 
 Let us now define 
 $\tilde v(\omega)= v_0 (\omega) \chi_{\Omega_1}(\omega)$. Taking into account \eqref{contr_det} we obtain
 \begin{align*}
    (A\tilde v+ K_C C\tilde v,\tilde v)_{L^2(\Omega;\R^n)}< \alpha_C\|\tilde v\|_{L^2(\Omega;\R^n)}^2,
\end{align*}
    which contradicts Assumption \ref{A0}.

    Reciprocally, in order to prove the reverse implication, it is enough to note that for almost every $\omega \in \Omega$ there exists $K_C(\omega)$ such that the inequality 
     \begin{align*}
        ((A(\omega)+ K_C(\omega) C(\omega))v_0(\omega),v_0(\omega))_{\R^n)} \geq  \alpha\|v_0(\omega)\|_{\R^n}^2,
 \end{align*}
 holds with a parameter independent constant $\alpha$.
 Integrating the last equation with respect to $\Omega$ we get the desired inequality \eqref{A0_ineq}. 
\end{proof}

Comparing the last result with Lemma \ref{assum_interp}, we see that Assumption \ref{A0} required for the turnpike property of simultaneous optimal control is weaker than  Assumption \ref{A2} used in this paper. Unlike the relation between averaged and simultaneous controllability, where the first notion is weaker, here we have the opposite kind of relation: the turnpike property for averaged optimal control is stronger than the one for simultaneous optimal control. This was to be expected since, in our case, we only penalized the average of the observations, contrary to what was done in \cite{HLZ23}.

\medbreak
    
$\bullet$ It is a well-known result that condition $(A,B)$-stabilizability and $(A,C)$-detectability are related to the exponential stability of the operators $A^* + B^* K_B^*$ and $A+K_C C$, respectively. In this sense, for parameter-dependent systems, hypotheses such as the following ones give us an exponential decay.
\begin{assumption}\label{complementary_assump1}
    There exists a
feedback operator $K_C\in C^0(\Omega,\mathcal{L}(\R^n))$ and $K_B \in C^0(\Omega, \mathcal{L} ( \R^m ; \R^n))$ uniformly bounded with respect to $\omega\in \Omega$, and a positive constants $\alpha_C,\alpha_B>0$ such that
\begin{align}\label{eq:ineq_1_extra_assum}
    (Av+ K_C Cv,\E[v])_{L^2(\Omega;\R^n)}\geq \alpha_C\|\E[v]\|_{\R^n}^2,\quad \text{for every }v\in L^2(\Omega;\R^n),
\end{align}
and 
\begin{align*}
    (A^*v+ K_B B^*v,\E[v])_{L^2(\Omega;\R^n)}\geq \alpha_B\|\E[v]\|_{\R^n}^2,\quad \text{for every }v\in L^2(\Omega;\R^n).
\end{align*}
\end{assumption}
Assumption \ref{complementary_assump1} is particularly interesting since inequality \eqref{eq:ineq_1_extra_assum} ensures that the average solution of system 
\begin{align*}
\begin{cases}
x_t(t,\omega)+A(\omega)x(t,\omega)+K_C(\omega)C(\omega)x(t,\omega)=0,\quad t>0,\\
    x(0,\omega)=x_0(\cdot)\in L^2(\Omega;\R^n)
\end{cases}
\end{align*}
exponential decay, that is,
\begin{align*}
    \|\E[x(t)]\|^2_{\R^n}\leq e^{-\alpha_C t}\|\E[x_0]\|^2_{\R^n}
\end{align*}
Thus, an interesting future research could be the analysis of the turnpike property, both integral and exponential, under the previous assumptions. Observe that these assumptions are weaker than uniformly detectability and stability. 

\medbreak

$\bullet$ Let us observe that, Assumption \ref{A1} is stronger than Assumption \ref{A0}. To illustrate this, consider the following $1-$dimensional example. Let us consider the discrete probability space $(\Omega,\mathcal{F},\mathbb{P})$ with $\Omega=\{\omega_1,\omega_2\}$ and $\mathbb{P}$ the Bernoulli distribution with parameter $p=1/2$. Then, we define random variables
\begin{align*}
    A(\omega)=\begin{cases}1, &\text{if }\omega=\omega_1,\\
    -1 &\text{if }\omega=\omega_2,        
    \end{cases}\quad\text{and}\quad C(\omega)=\begin{cases}0 &\text{if }\omega=\omega_1,\\
    -1 &\text{if }\omega=\omega_2.        
    \end{cases}
\end{align*}
Taking $B=1$, we can introduce the optimal control \eqref{eq:evol_funtional}-\eqref{eq:evol_equation}, obtaining an optimal pair $(x(\omega),u)$. Observe that if $\omega=\omega_1$, we define a system that only depends on $\omega_1$. Similar to $\omega=\omega_2$. Then, we have that $\mathbb{P}[C(\omega)x(\omega)=C(1)x(2)]=\mathbb{P}[C(\omega)x(\omega)=C(2)x(1)]=0$. 

Now, consider a random variable $K_C$, and observe that
\begin{align*}
    (Av+K_CCv,v)_{L^2(\Omega;\R)}&= \frac{1}{2}\big((A(\omega_1)v(\omega_1)+K_C(\omega_1)C(\omega_1)v(\omega_1))v(\omega_1)\\
    &\hspace{2cm}+(A(\omega_2)v(\omega_2)+K_C(\omega_2)C(\omega_2)v(\omega_2))v(\omega_2)\big)\\
    &=\frac{1}{2}\big(|v(\omega_1)|^2 -|v(\omega_2)|^2 -K_C(\omega_2)|v(\omega_2)|^2\big).
\end{align*}
Therefore, taking $K_C(\omega)$ such that $K_C(\omega_1)=0$ and $K_C(\omega_2)=-2$ we have
\begin{align*}
    (Av+K_CCv,v)_{L^2(\Omega;\R)}\geq \frac{1}{2}(|v(\omega_1)|^2+|v(\omega_2)|^2)=\|v\|_{L^2(\Omega;\R)},
\end{align*}
and Assumption \ref{A0} holds. 

On the other hand, observe that 
\begin{align*}
     (Av+K_C&\E[K_C Cv],v)_{L^2(\Omega;\R)}\\
     &=\frac{1}{2}\left(A(\omega_1)v(\omega_1)+\frac{K_C(\omega_1)}{4}\left(K_C(\omega_1)C(\omega_1)v(\omega_1)+K_C(\omega_2)C(\omega_2)v(\omega_2)\right)\right)v(\omega_1)\\
    &\hspace{0.5cm}+\frac{1}{2}\left(A(\omega_2)v(\omega_2)+\frac{K_C(\omega_2)}{4}\left(K_C(\omega_1)C(\omega_1)v(\omega_1)+K_C(\omega_2)C(\omega_2)v(\omega_2)\right)\right)v(\omega_2)\\
    &=\frac{1}{2}\big(|v(\omega_1)|^2-\frac{K_C(\omega_1)}{4}K_C(\omega_2)v(\omega_2)v(\omega_1)-|v(\omega_2)|^2-\frac{|K_C(\omega_2)|^2}{4}|v(\omega_2)|^2.
\end{align*}
Then, we can take $K_C(\omega_1)=0$ and recover the $|v(\omega_2)|$ in order to have the $\|v\|_{L^2(\Omega;\R)}$ norm. However, there is no real value for $K_C(\omega_2)$ such that we can obtain the $|v(\omega_2)|$ with a positive sign. Consequently, Assumption \ref{A1} does not hold.

%%%%%%%%%%%%%%%%%%%%%%%%%

\section*{Acknowledgments}

M. Hernandez has been funded by the Transregio 154 Project, “Mathematical Modelling, Simulation, and Optimization Using the Example of Gas Networks” of the DFG, project C07, and the fellowship "ANID-DAAD bilateral agreement". M. Lazar was supported in part by the Croatian Science Foundation under the project Conduction, IP-2022-10-5191, as well as by the German Federal Ministry of Education and Research (BMBF)
and the Croatian Ministry of Science and Education under the DAAD grant 57654073; Uncertain Data in Control of PDE Systems. S. Zamorano was supported by the Alexander von Humboldt Foundation by the Research Fellowship for Experienced Researchers. 

Part of this research was carried out while the third author visited the FAU DCN-AvH Chair for Dynamics, Control, Machine Learning and Numerics at the Department of Mathematics at FAU, Friedrich-Alexander-Universit\"{a}t Erlangen-N\"{u}rnberg (Germany). He would like to thank the members of this institution for their kindness and warm hospitality. 

%
%\bibliographystyle{abbrv} 
%\bibliography{biblio.bib}

\end{document}